\documentclass[12 pt]{amsart}

\usepackage{amsthm}
\usepackage{amsmath}
\usepackage{amssymb}
\usepackage{latexsym}
\usepackage{url}
\usepackage{graphicx}
\usepackage{bmpsize}
\usepackage[margin=3cm]{geometry}
\usepackage{enumerate}

\author{Felix Goldberg}
\address{Caesarea-Rothschild Institute, University of Haifa, Haifa, Israel}
\email{felix.goldberg@gmail.com}

\title{Chip-firing may be much faster than you think}
\date{November 24, 2014}

\newtheorem{thm}{Theorem}[section]
\newtheorem{cor}[thm]{Corollary}

\newtheorem{lem}[thm]{Lemma}

\newtheorem{defin}[thm]{Definition}
\newtheorem{expl}[thm]{Example}

\newtheorem{rmrk}[thm]{Remark}

\DeclareMathOperator{\Tr}{Tr}

\theoremstyle{example}

\newtheoremstyle{example_contd}
{\topsep} {\topsep}%
{\upshape}
{}
{\bfseries\scshape}
{.}
{1em}
{\thmname{#1} \thmnumber{ #2}\thmnote{#3} (continued)}

\theoremstyle{example_contd}

\begin{document}

\begin{abstract}
A new bound (Theorem \ref{thm:main}) for the duration of the chip-firing game with $N$ chips on a $n$-vertex graph is obtained, by a careful analysis of the pseudo-inverse of the discrete Laplacian matrix of the graph. This new bound is expressed in terms of the entries of the pseudo-inverse. 

It is shown (Section 5) to be always better than the classic bound due to Bj{\"o}rner, Lov\'{a}sz and Shor. In some cases the improvement is dramatic. 

For instance: for strongly regular graphs the classic and the new bounds reduce to $O(nN)$ and $O(n+N)$, respectively. For dense regular graphs - $d=(\frac{1}{2}+\epsilon)n$ - the classic and the new bounds reduce to $O(N)$ and $O(n)$, respectively.

This is a snapshot of a work in progress, so further results in this vein are in the works.
\end{abstract}



\thanks{{This research was supported by the Israel Science Foundation (grant number 862/10.)}}

\maketitle

\section{Introduction}
Consider the following solitaire game: some chips are placed on the vertices $1,2,\ldots,n$ of a graph $G$ so that vertex $i$ has degree $d_{i}$ and receives $a_{i}$ chips. Then the player performs a series of moves which are called ``firings''. In each such move she selects a vertex for which $a_{i} \geq d_{i}$ and moves one chip from $i$ to each of $i$'s neighbours. If there are no possible moves, the game ends. It is also possible for the game to enter an infinite loop, when a quondam position comes up again.

This rather innocuous-sounding game is actually laden with a staggering amount of very deep properties. Papers on it may be found under different names in journals devoted to mathematics, physics and theoretical computer science.
We shall now very briefly mention some of these connections and then proceed to outline our new contribution.

\subsection{A portal to bibliography}
The chip-firing game, in the form described above, was introduced by Bj{\"o}rner, Lov\'{a}sz and Shor \cite{BjoLovSho91} in 1991. 
A physicist would recognize this as a fixed-energy variant of the ``sandpile model'' \cite{Bak,Dha90} which has been suggested as a possible simple model for the emergence of power laws in various of natural phenomena, for instance earthquakes. 

The chip firing game is related to potential theory on graphs and to arithmetic geometry. We refer the reader to \cite{BakSho13,Big97,Lor08,Mer05} for more on this. A close relation to random walks is described in \cite{LovWin95}. For a description of the relation between the chip firing game to lattice theory we refer to \cite{Goles_survey} and for the chip firing game in the role of a universal computer to \cite{GolMar97}.

Finally, we suggest \cite{Backman_thesis,BakSho13,ChuEll02,Goles_survey,Rotor,Mer05} as possible entry points to the quite vast literature on the subject.


\section{A remarkable property}\label{sec:bls}
\begin{thm}\cite{BjoLovSho91}\label{thm:bls}
Given a connected graph and an initial distribution of chips, either every legal game can be continued indefinitely, or every legal game terminates after the same number of moves with the same final position. The number of times a given vertex is fired is the same in every legal game.

\end{thm}


Theorem \ref{thm:bls} may feel rather surprising at first glance. While the inherent determinism of the game might not be that surprising, all things considered, the fact that the game always takes the same exact number of moves is harder to stomach. 

However, the derivation in \cite{BjoLovSho91} elucidates all: the set of legal gameplays, considered as sequences, forms an 'antimatroid with repetitions' and the terminating games correspond to the bases.Thus we see that the surprising claim about game duration is just good old equicardinality of bases in heavy disguise.

\subsection{Goal of the paper}
In this paper we shall concentrate on the problem of obtaining structural bounds on the possible duration of a game. This problem has received some attention and a number of bounds are known. 

However, computer simulations (cf. Section \ref{sec:computer}) show that the extant bounds are often far too pessimistic and that the game tends to end much faster than predicted by them. Therefore, there is a need to develop new bounds which will be closer to the actual values of the game duration.

\section{The problem of game duration}
Let us pause to fix notation for the rest of the paper. We will be dealing with a connected graph $G$ with $n$ vertices and $m$ edges. The number of chips in play will be denoted $N$. Let us assume that the game terminates in $s$ moves and let $x_{i}$ be the number of times that vertex $i$ has been fired during the game. 

A result from \cite{BjoLovSho91} says that in a terminating game we must necessarily have $N \leq 2m-n$. Furthermore, if $N<m$ then \emph{every} initial distribution with $N$ chips will terminate. 

The first bound on game duration was given by Tardos \cite{Tar88}:
\begin{thm}\cite{Tar88}\label{thm:tardos}
Suppose that the diameter of the graph $G$ is $D$. Then
$$s \leq nND.$$
\end{thm}

This result is sometimes quoted in a less precise from as $s=O(n^{4})$ which holds since clearly $D=O(n)$ and $N < 2m=O(n^{2})$. Tardos \cite{Tar88} also shows an example of a game which does take $O(n^{4})$ moves to terminate. Note that for directed graphs it was proved by Eriksson \cite{Eri91} that no such polynomial bound is possible.


Recall that the \emph{Laplacian} matrix $L$ (cf. \cite{Mer94,Moh97}) of the graph $G=(V,E)$ whose vertices are labelled $\{1,2,\ldots,n\}$ is:
$$
L_{ij} = \begin{cases}
-1  & \text{, if } (i,j) \in E \\
0   & \text{, if } (i,j) \notin E \text{ and } i \neq j\\
d_{i} & \text{, if } i=j.
\end{cases}
$$

Denote by $\mathbf{a} \in \mathbb{R}^{n}$ the vector representing the chip distribution at some moment and by $\mathbf{a}^{'} \in \mathbb{R}^{n}$ the vector representing the distribution of the chips after firing vertex $i$. We see that $\mathbf{a}^{'}$ is obtained from $\mathbf{a}$ by subtracting the $i$th column of $L$. Therefore we have the following important fact, observed first in \cite{BjoLovSho91}:

\begin{thm}\cite{BjoLovSho91}\label{thm:chip}
Suppose that a terminating game is played on $G$. Let $\mathbf{a} \in \mathbb{R}^{n}$ represent the initial chip distribution and $\mathbf{b} \in \mathbb{R}^{n}$ the final chip distribution. Then
\begin{equation}\label{eq:chip}
Lx=\mathbf{a}-\mathbf{b}.
\end{equation}
\end{thm}

Bj{\"o}rner, Lov\'{a}sz and Shor \cite{BjoLovSho91} derived from \eqref{eq:chip} the following bound:

\begin{thm}\cite{BjoLovSho91}\label{thm:blsb}
Let the eigenvalues of $L$ be $0<\lambda_{2} \leq \ldots \leq \lambda_{n}$. Then 
$$s \leq \frac{2nN}{\lambda_{2}}.$$
\end{thm}

The smallest non-trivial Laplacian eigenvalue $\lambda_{2}$ is often called the \emph{algebraic connectivity} of $G$ (cf. \cite{Fie73,Abr07}).

The bounds of Theorems \ref{thm:tardos} and \ref{thm:blsb} are elegant and crisp. However, as we will now show by examples, they tend to severely overestimate the number of moves required for the game.

\subsection{A few exemplary games}\label{sec:computer}

The chip-firing game can be easily implemented on a computer and because of Theorem \ref{thm:bls} we need not worry about the choice of vertices to fire at each stage - we can just choose any vertex we like. 

Therefore, we can examine what happens when we play the game on three graphs: the (in)famous Petersen graph, the Schl\"{a}fli graph \footnote{The complement of the collinearity graph of $GQ(2,4)$} and the Paley graph with $109$ vertices. In all three cases all $N$ chips were initially placed on a single vertex. 

The table below compares the actual number of moves expended in the games with the foregoing upper bounds.

\medskip

\begin{table}[here]
\caption{Upper bounds on the number of moves}
\begin{tabular}{|c|l|l|l|l|l|}

\hline

Name \ & $n$ \ & $N$\ & Theorem \ref{thm:tardos}\ & Theorem \ref{thm:blsb}  \ & Actual duration \\ [0.5ex] 

\hline

Petersen & $10$ & $14$ & $280$ & $140$ & $8$\\ 
Schl\"{a}fli & $27$ & $215$ & $11610$ & $967$ & $13$\\ 
Paley($109$) & $109$ & $2900$ & $632200$ & $12828$ &  $53$\\
\hline
\end{tabular}
\end{table}
\medskip

Clearly, there is some serious overestimation here. Our main result will be to provide a new bound for the class of \emph{strongly regular} graphs, to which these three graphs belong. Roughly speaking, instead of a $O(nN)$ bound we will be able to obtain a $O(n+N)$ bound.

\section{Main Implicit Bound}\label{sec:mp}

The quantity that we are interested in is $$s=\sum_{i=1}^{n}{x_{i}}.$$ If the matrix $L$ were invertible, we could have solved for $x$ as $$x=L^{-1}(\mathbf{a}-\mathbf{b})$$ and then derived estimates on $s$. However, as is well-known, the rank of a Laplacian matrix of a graph with $c$ connected components is $n-c$ and thus our $L$ is singular. 

Nevertheless, in \cite{BjoLovSho91} it was observed that it is possible to use the so-called Moore-Penrose pseudo-inverse of $L$ in a similar way. We shall only give here the briefest of introductions to pseudo-inversion, referring the interested reader to the books \cite{BenIsraelGreville,CampbellMeyer}.

\subsection{A crash course on generalized inversion}
Let $A$ be a $m \times n$ complex matrix. Then there exists a unique $n \times m$ matrix $X$ such that:
\begin{enumerate}[(i)]
\item
$AXA=A$.
\item
$XAX=X$.
\item
$AX$ and $XA$ are Hermitian.
\end{enumerate}

Clearly, if $A$ is square and inverible, then $X=A^{-1}$. The matrix $X$ is denoted $A^{\dagger}$ and called the Moore-Penrose pseudo-inverse of $A$. 

The foregoing axiomatic definition does not tell us how to compute $A^{\dagger}$ but fortunately there is another characterization which does. Let $$A=U\Sigma V^{*}$$ be the singular value decomposition of $A$, where $U$ and $V$ are orthogonal matrices of appropriate orders and $\Sigma$ is a $m \times n$ matrix which is ``generalized diagonal'', in the sense that $\Sigma_{ij}=0$ for $i \neq j$. 

It is easy to see that $\Sigma^{\dagger}$ is obtained by replacing every non-zero entry in $\Sigma$ by its inverse.
Then we have
\begin{equation}\label{eq:svd}
A^{\dagger}=V\Sigma^{\dagger}U^{*}.	
\end{equation}

If the matrix $A$ has a spectral decomposition we can deduce from \eqref{eq:svd} a very convenient representation of the pseudo-inverse: 
\begin{equation}\label{eq:spec}
A=\sum_{i=1}^{k}{\lambda_{i}E_{i}}, \quad A^{\dagger}=\sum_{i=1}^{k}{\lambda_{i}^{-1}E_{i}},
\end{equation}
where the $\lambda_{i}$ are the distinct nonzero eigenvalues of $A$ and the $E_{i}$ are the corresponding orthogonal projections.


There is quite a large number of papers devoted to describing $L^{\dagger}$ combinatorially. We may point out \cite{Bap97} for trees and \cite{CheSha98} for weighted multigraphs as good entrance points to this subject.

\subsection{An implicit bound for any graph}
Before we state the new bound, we need to collate two very useful observations made by the pioneers of the subject. 

\begin{lem}\cite{Tar88}
Suppose that a terminating game was played on $G$. Then at least one vertex $k$ has not fired during the game, that is $x_{k}=0$.
\end{lem}

Let us denote the standard basis vectors of $\mathbb{R}^{n}$ as $e_{1},e_{2},\ldots,e_{n}$.

\begin{lem}\cite{BjoLovSho91}
Suppose that a terminating game was played on $G$. Let $k$ be a vertex such that $x_{k}=0$. Then
\begin{equation}\label{eq:s}
s=-ne_{k}L^{\dagger}(\mathbf{a}-\mathbf{b}).
\end{equation}
\end{lem}

Theorem \ref{thm:blsb} was in fact derived in \cite{BjoLovSho91} from \eqref{eq:s} by what in effect amounts to ``bounding'' the spectral decomposition of $L^{\dagger}=\sum_{i=1}^{k}{\lambda_{i}^{-1}E_{i}}$ by $\frac{1}{\lambda_{2}}\sum_{i=1}^{k}{E_{i}}$. However, it is possible to get stronger bounds on $s$ if we delve more deeply into the actual entries of $L^{\dagger}$. Our next theorem, which is the first new main result of the paper, provides a bound for $s$ in terms of the entries of $L^{\dagger}$.

To state our theorem it will be convenient to introduce the following notation:
$$
f=\max_{i=1}^{n} \{L^{\dagger}_{ii}\},  \quad o=\max_{i\neq j} \{|L^{\dagger}_{ij}|\}.
$$

\begin{lem}\label{lem:fo}
$f \geq o$.
\end{lem}
\begin{proof}
Suppose that $o=|L^{\dagger}_{ij}|$. Consider the $2 \times 2$ submatrix $H$ based on the $i$th and $j$th lines of $L^{\dagger}$. Since $H$ is positive semidefinite we have $0 \leq |H|=L^{\dagger}_{ii}L^{\dagger}_{jj}-o^{2} \leq f^{2}-o^{2}$.

\end{proof}

\begin{thm}[Main Implicit Bound]\label{thm:main}
Suppose that a terminating game was played on $G$. Suppose that the maximum degree of $G$ is $\Delta$. Let $k$ be a vertex such that $x_{k}=0$. Then
\begin{equation}\label{eq:mine}
s \leq n\Big(f(\Delta-1)  +o(2N-\Delta+1) \Big).
\end{equation}
\end{thm}
\begin{proof}
Denote $\mathbf{d}=\mathbf{a}-\mathbf{b}$. We have from \eqref{eq:s} that $s$ equals $-n$ times the scalar product of the $k$th row of $L^{\dagger}$ with the vector $\mathbf{d}$. Now, we observe that as the vertex $k$ had not been fired, we must have $\mathbf{a}_{k} \leq \Delta-1$ or else $k$ would have fired at some stage. On the other hand, $\mathbf{b}_{k} \leq \Delta-1$ as the game terminated. Therefore $|\mathbf{d}_{k}| \leq \Delta-1$. On the other hand $\sum_{i \neq k}{|\mathbf{d}_{i}|} \leq 2N-|\mathbf{d}_{k}|$. 

We can combine our observations to write:
$$
s=-n(L^{\dagger}_{kk}\mathbf{d}_{k}+\sum_{i \neq k}{L^{\dagger}_{ki}\mathbf{d}_{k}}) \leq  n(f(|\mathbf{d}_{k}|)+o(2N-|\mathbf{d}_{k}|)) \leq n(f(\Delta-1)+o(2N-\Delta+1)).
$$

\end{proof}

\section{Improving the Bj{\"o}rner-Lov\'{a}sz-Shor bound}
Since $f \geq o$, the Main Implicit Bound has the following corollary:
\begin{cor}\label{cor:f}
$$
s \leq 2nNf.
$$
\end{cor}

We now observe that even this corollary is strong enough to imply Theorem \ref{thm:blsb}:
\begin{cor}[Bj{\"o}rner-Lov\'{a}sz-Shor]
$$
s \leq \frac{2nN}{\lambda_{2}}.
$$
\end{cor}
\begin{proof}
Schur's majorization theorem (cf. \cite[Theorem 4.3.26]{HornJohnson}) tells us that the largest diagonal entry $f$ of $L^{\dagger}$ is bounded from above by the largest eigenvalue $\frac{1}{\lambda_{2}}$ of $L^{\dagger}$.
\end{proof}

As a warm-up let us now obtain a modest improvement upon Theorem \ref{thm:blsb} for vertex-transitive graphs:
\begin{thm}
Let $G$ be a vertex-transitive graph. Then $$s \leq \frac{2(n-1)N}{\lambda_{2}}.$$
\end{thm}
\begin{proof}
All the diagonal entries of $L^{\dagger}$ are equal in this case to $f$. Thus we have:
$$
f=\frac{\Tr(L^{\dagger})}{n}=\frac{\sum_{i=2}^{n}{\frac{1}{\lambda_{i}}}}{n} \leq \frac{n-1}{n} \cdot \frac{1}{\lambda_{2}}.
$$
\end{proof}

Now let us proceed to improve upon Theorem \ref{thm:blsb} for all graphs, using a more complicated estimate for $f$. This will require some setting-up. First we recall a well-known formula for $L^{\dagger}$ which can be deduced from \eqref{eq:spec}. $J$ denotes the all-ones matrix and $c \neq 0$:
\begin{equation}\label{eq:t}
L^{\dagger}=(L+c J)^{-1}-\frac{1}{c n^{2}}J.
\end{equation}

\begin{thm}\cite[Theorem 5.1]{GolMeu93}\label{thm:gm}
Let $A$ be a positive definite matrix whose eigenvalues are contained in the interval $[a,b]$, $a>0$. Then
\begin{equation}\label{eq:gm}
A^{-1}_{ii} \leq \frac{a+b-a_{ii}}{ab}.
\end{equation}
\end{thm}

Now we can prove our result:
\begin{thm}\label{thm:bls_better}
For any graph $G$ we have:
$$
s \leq 2nN \cdot \frac{\lambda_{2}+\frac{\lambda_{n}(n-1)}{n}-\delta}{\lambda_{2}\lambda_{n}}.
$$
\end{thm}
\begin{proof}
Let $T=L+cJ$ with $c=\frac{\lambda_{n}}{n}$. The spectrum of $T$ is identical to that of $L$, except for the zero which becomes $cn=\lambda_{n}$. Therefore, the smallest and largest eigenvalues of $T$ are $\lambda_{2}$ and $\lambda_{n}$, respectively. We use \eqref{eq:t} and \eqref{eq:gm} to write:
$$
L^{\dagger}_{ii} \leq T^{-1}_{ii} \leq \frac{\lambda_{2}+\lambda_{n}-d_{i}-c}{\lambda_{2}\lambda_{n}}=\frac{\lambda_{2}+\frac{\lambda_{n}(n-1)}{n}-d_{i}}{\lambda_{2}\lambda_{n}}.
$$
Therefore 
\begin{equation}\label{eq:f_soph}
f \leq \frac{\lambda_{2}+\frac{\lambda_{n}(n-1)}{n}-\delta}{\lambda_{2}\lambda_{n}}
\end{equation}
\smallskip
and the conclusion follows immediately from Corollary \ref{cor:f}.
\end{proof}

Theorem \ref{thm:bls_better} is always stronger than Theorem \ref{thm:bls} because  $\lambda_{2} \leq \delta$ by a classic result of Fiedler \cite{Fie73}.

\section{Beating $o \leq f$ for dense regular graphs}
The explicit results of the previous section were obtained by substituting certain estimates for $f$ into the Main Implicit Bound. We will obtain even better explicit results by estimating $o$ in its own right, improving upon the simple $o \leq f$. This is rather difficult to do, but nevertheless we will now show a way of deriving an estimate for $o$ in the case of dense regular graphs.


Let $A$ be a $n \times n$ matrix. Denote $R_{i}(A)=\sum^{n}_{j=1,j \neq i}{|a_{ij}|}$ and $\sigma_{i}(A)=\frac{R_{i}(A)}{|a_{ii}|}$. If $\sigma_{i}(A) \leq 1$ for all $i=1,2,\ldots,n$ we say that the matrix $A$ is \emph{diagonally dominant} and if $\sigma_{i}(A) < 1$ for all $i$, then $A$ is said to be \emph{strictly diagonally dominant} or \emph{SDD}. Clearly, the Laplacian matrix $L$ is diagonally dominant but not SDD.

The inverse of an SDD matrix is not necessarily SDD. However, Ostrowski \cite{Ost52} observed that a weak form of diagonal dominance does carry over to the inverse. Namely: if $A$ is SDD and $B=A^{-1}$, then $|b_{ji}| \leq \sigma_{j}(A) |b_{ii}|$. 

Ostrowski's theorem was later rediscovered by Yong and Wang \cite{YongWang99} whose work has rekindled interest in such results, in the context of the so-called Fiedler-Markham cojecture. However, they are embedded as technical lemmae in various papers and are not readily available to the casual peruser. We are going to use one such result which is particularly simple to apply while being rather effective:

\begin{thm}\cite[Theorem 2.4]{LiHuaSheLi07}\label{thm:diagdom}
Let $A$ be SDD. Let $B=A^{-1}=({b}_{ij})$. Then it holds that:
\begin{equation}\label{eq:dom}
|{b}_{ji}| \leq \max_{l \neq i} \left\{ \frac{|a_{li}|}{|a_{ll}|-\sum_{k \neq l,i}|a_{lk}|} \right\} |{b} 	_{ii}|, \quad \textit{for all} \quad j \neq i.
\end{equation}
\end{thm}

We are now in a position to prove:

\begin{thm}\label{thm:2d}
Let $G$ be a connected $d$-regular graph on $n$ vertices. If $d>\frac{n}{2}-1$, then
$$
o \leq \frac{f}{2d-n+3}+\frac{2}{n^2}.
$$
\end{thm}
\begin{proof}
Let $T=L+J$. We claim that $T$ is an SDD matrix. Indeed, $t_{ii}=d+1$ and $R_{i}(T)=n-d-1$. Denote 
$$
f_{T}=\max_{i=1}^{n} \{T^{-1}_{ii}\},  \quad o_{T}=\max_{i\neq j} \{|T^{-1}_{ij}|\}.
$$
As a consequence of \eqref{eq:spec} we have
$$
f=f_{T}-\frac{1}{n^{2}}, \quad o \leq o_{T}+\frac{1}{n^{2}}.
$$

Since the only off-diagonal entries of $T$ are $0$ and $1$ we can apply Theorem \ref{thm:diagdom} with $t_{li}=1$, $t_{ll}=d+1$, and $\sum_{k \neq l,i}{|t_{lk}|}=n-d-2$ to obtain:
$$
o_{T} \leq \frac{f_{T}}{2d-n+3}.
$$ 
Finally,
$$
o \leq o_{T}+\frac{1}{n^2} \leq \frac{f_{T}}{2d-n+3} +\frac{1}{n^2}=\frac{f+\frac{1}{n^{2}}}{2d-n+3} +\frac{1}{n^2} \leq \frac{f}{2d-n+3}+\frac{2}{n^2}.
$$
\end{proof}

Theorem \ref{thm:2d} is a \emph{relative} estimate for $o$ in the sens that it depends on $f$. We shall see in the next section a situation in which extra combinatorial structure allows us to provide \emph{absolute} estimates for $o$. 

In the presentation of the next result we aim rather more for elegance of expression that for the utmost optimization of lower-order terms and so we use $2N$ instead of $2N-d+1$ and estimate $f$ via $\frac{1}{\lambda_{2}}$ rather than via the sharper but more cumbersome expression in \eqref{eq:f_soph}.

\begin{thm}\label{thm:reg}
Let $G$ be a connected $d$-regular graph on $n$ vertices and suppose that $d=(\frac{1}{2}+\epsilon)n, \quad 0<\epsilon<\frac{1}{2}$. Then 
$$
s \leq \frac{n(d-1)}{\lambda_{2}}+\frac{N}{\lambda_{2}\epsilon n}+\frac{4N}{n}.
$$
\end{thm}
\begin{proof}
Apply the Main Implicit Bound in the simplified form $s \leq n(f(d-1)+2No)$, together with Theorem \ref{thm:2d}. Note that $2d-n+3=2\epsilon n+3 \geq 2\epsilon n$.
\end{proof}

We can offer a probabilistic comparison of Theorem \ref{thm:reg} with Theorem \ref{thm:blsb} based on a result by Juh\'{a}sz \cite{Juh91}:
\begin{thm}\cite[Theorem 2]{Juh91}
Let $G(n,p)$ be a random graph. Then the algebraic connectivity $\lambda_{2}$ of $G$ satisfies for any $\epsilon>0$:
$$
\lambda_{2}(G)=pn+o(n^{\frac{1}{2}+\epsilon}) \text{ in probability.}
$$
\end{thm}
Therefore we are justified in writing $\frac{1}{\lambda_{2}}=O(n)$, in the probabilistic sense. Also $d=\Theta(n)$ and $N=O(n^{2})$ and we see that the BLS bound reduces to $O(N)$ while our new bound reduces to $O(n)$ (probabilistically). Since for the game to even start we must have $N \geq d=\Omega(n)$, our bound is always at least as good as the BLS one.

%

\section{strongly regular graphs}
\begin{defin}\cite{Bose63}\label{dfn:srg}
A \emph{strongly regular graph} with parameters $(n,k,a,c)$ is a
$k$-regular graph on $n$ vertices such that any two adjacent
vertices have $a$ common neighbours and any two non-adjacent
vertices have $c$ common neighbours.
\end{defin}

If $G$ is strongly regular with parameters $(n,k,a,c)$ we shall also write compactly that $G$ is $SRG(n,k,a,c)$. Nice expositions of the theory of strongly regular graphs may be found in \emph{e.g.}, \cite{Spectra_BH,Links,AGT2,Course}. The graphs discussed in Section \ref{sec:computer} are all strongly regular: Petersen is $SRG(10,3,0,1)$, the Schl\"{a}fli graph is $SRG(27,16,10,8)$ and Paley(109) is $SRG(109,54,25,26)$.

Recall that the adjacency matrix $A$ of a strongly regular graph has exactly three distinct eigenvalues: $k,\theta,\tau$, with $\theta>0$ and $\tau<0$. It is well known (cf. \cite[p. 220]{AGT2}) that the eigenvalues $\theta,\tau$ are given by:
$$
\theta=\frac{(a-c)+\sqrt{\Delta}}{2}, \quad \tau=\frac{(a-c)-\sqrt{\Delta}}{2},
$$
where $\Delta=\sqrt{(a-c)^{2}+4(k-c)}$.

\medskip


For the duration of this section $G$ will refer to a connected $SRG(n,k,a,c)$. We will also denote $d=a-c$ and since $c=0$ would have implied a disconnected $G$ (cf. \cite[p. 218]{AGT2}) we freely assume that $c \geq 1$ and thus $d \leq a-1$.

\subsection{Auxiliary results}
We now collate a number of facts that will be used in the proof of our new result, Theorem \ref{thm:srg}.
\begin{lem}\label{lem:s1}
$$\theta+\tau=a-c=d.$$
\end{lem}

\begin{lem}\label{lem:s2}
$$(k-\theta)(k-\tau)=k(k-d-1)+c.$$
\end{lem}
\begin{proof}
$$
(k-\theta)(k-\tau)=(k-\frac{d+\sqrt{\Delta}}{2})(k-\frac{d-\sqrt{\Delta}}{2})=\frac{(2k-d)^{2}-\Delta}{4}=\frac{4k^{2}-4kd-4(k-c)}{4}.
$$
\end{proof}
\begin{lem}\cite[p. 244]{AGT2}\label{lem:s3}
$$(k-\theta)(k-\tau)=nc.$$
\end{lem}
\begin{lem}[Taylor, Levingstone]\cite[p. 7]{BCN}\label{lem:taylor} 
Suppose that $G$ is connected and $G \neq K_{n}$. Then $$k \geq 2a-c+3$$ and equality holds if and only if $G$ is a pentagon.
\end{lem}
\begin{lem}\label{lem:s4} 
Suppose that $G$ is connected and $G \neq K_{n},C_{5}$. Then $$d \leq \frac{k-5}{2}.$$
\end{lem}
\begin{proof}
By Lemma \ref{lem:taylor} it follows that $k \geq 2a-c+4$ under our assumptions. Therefore:
$$
k \geq 2a-c+4=d+4+a \geq d+4+d+1=2d+5.
$$
\end{proof}
\begin{lem}\cite[p. 219]{AGT2}\label{lem:s5}
$$k(k-a-1)=(n-k-1)c.$$
\end{lem}

\begin{lem}\label{lem:s6}
$$2(n-k) \geq -d.$$
\end{lem}
\begin{proof}
From Lemma \ref{lem:s5} we have that $(n-k)=\frac{k(k-a-1)}{c}+1$. Therefore our claim is equivalent to
$$
2k(k-a-1) \geq c(c-a-2)
$$
which always holds by elementary algebra.
\end{proof}

\subsection{Chip-firing on $SRG$ is fast}
Theorems \ref{thm:tardos} and \ref{thm:blsb} imply a $O(nN)$ bound on game duration for a strongly regular graph, but by using our method it can be improved considerably to $O(n+N)$. Specifically we have:

\begin{thm}\label{thm:srg}
Let $G$ be a connected $SRG(n,k,a,c)$. Then $$s \leq n\frac{k-1}{k-2}+\frac{2(2N-k+1)}{c}.$$
\end{thm}
\begin{proof}
The claim will follow from Theorem \ref{thm:main} and the following two estimates:
\begin{equation}\label{eq:f_srg}
f \leq \frac{1}{k-2},
\end{equation}
\begin{equation}\label{eq:o_srg}
o \leq \frac{2}{nc}.
\end{equation}

The three orthogonal eigenprojections of $A$ are:
\begin{equation}\label{eq:projs}
E_{0}=\frac{J}{n}, \quad E_{1}=\frac{1}{\theta-\tau}\Big(A-\tau I-\frac{k-1}{n}J\Big), \quad E_{2}=\frac{1}{\tau-\theta}\Big(A-\theta I-\frac{k-\theta}{n}J\Big)
\end{equation}
and we have
$$
A=kE_{0}+\theta E_{1}+\tau E{2}.
$$
Since $E_{0}+E_{1}+E_{2}=I$ we infer the following expression for $L=kI-A$:
$$
L=(k-\theta)E_{1}+(k-\tau)E_{2}.
$$
Applying \eqref{eq:spec} we have
\begin{equation}\label{eq:ldag}
L^{\dagger}=\frac{1}{k-\theta}E_{1}+\frac{1}{k-\tau}E_{2}.
\end{equation}

Now a tedious but straightforward calculation from \eqref{eq:projs} and \eqref{eq:ldag} yields:
$$
L^{\dagger}_{ii}=\frac{k(n-2)-(n-1)(\theta+\tau)}{n(k-\theta)(k-\tau)}.
$$
Let us use this expression in conjunction with Lemmae \ref{lem:s1} and \ref{lem:s2} to estimate $f$:
$$
f=\frac{k(n-2)-(n-1)(\theta+\tau)}{n(k-\theta)(k-\tau)} \leq \frac{k-(\theta+\tau)}{(k-\theta)(k-\tau)}=\frac{k-d}{k(k-d-1)+c} \leq \frac{k-d}{k(k-d-1)}.
$$
Furthermore, Lemma \ref{lem:s4} enables us to complete the estimate:
$$
f \leq \frac{k-d}{k(k-d-1)}=\frac{1}{k}\Big(1+\frac{1}{k-d-1}\Big) \leq \frac{1}{k}\Big(1+\frac{1}{k-\frac{k-5}{2}-1}\Big)=\frac{k+5}{k(k+3)} \leq \frac{1}{k-2}.
$$
This proves estimate \eqref{eq:f_srg}.

Now we consider the off-diagonal entry $L^{\dagger}_{ij}$. Once again, we calculate it from \eqref{eq:projs} and \eqref{eq:ldag} but now there are two possible cases: when $i$ and $j$ are adjacent and when they are not. In the former case we have
$$
L^{\dagger}_{ij}=\frac{n-2k+\theta+\tau}{n(k-\theta)(k-\tau)}=\frac{n-2k+d}{n(k-\theta)(k-\tau)},
$$
and in the latter case
$$
L^{\dagger}_{ij}=\frac{\theta+\tau-2k}{n(k-\theta)(k-\tau)}=\frac{d-2k}{n(k-\theta)(k-\tau)}.
$$

The denominator of both expressions is equal to $n^{2}c$ by Lemma \ref{lem:s3}. We now claim that the numerators are bounded in absolute value by $2n$, which will prove estimate \eqref{eq:o_srg}. In fact, for the first expression we can even show the slightly stronger fact that the numerator's modulus is bounded by $n$. 

Indeed, $n-2k+d\leq n$  is obvious and $n-2k+d \geq -n$ is exactly the claim of Lemma \ref{lem:s6}. For the second expression, $d-2k \leq 2n$ is obvious and $d-2k \geq -2n$ is once again Lemma \ref{lem:s6}.
\end{proof}

Let us now reproduce Table 1 with two additional column indicating the new bounds of Theorems \ref{thm:main} and \ref{thm:srg}.

\medskip

\begin{table}[here]
\caption{Upper bounds on the number of moves}
\begin{tabular}{|c|l|l|l|l|l|l|l|}

\hline

Name \ & $n$ \ & $N$\ & Theorem \ref{thm:tardos}\ & Theorem \ref{thm:blsb} \ & Theorem \ref{thm:main} \ & Theorem \ref{thm:srg} \ & Actual $s$ \\ [0.5ex] 

\hline

Petersen & $10$ & $14$ & $280$ & $140$ & $24$ & $72$ & $8$\\ 
Schl\"{a}fli & $27$ & $215$ & $11610$ & $967$ & $81$ & $132$ & $13$\\ 
Paley($109$) & $109$ & $2900$ & $632200$ & $12828$ & $318$ & $536$ & $53$\\
\hline
\end{tabular}
\end{table}
\medskip

The improvement is palpable. 

\begin{expl}
For the Paley graph on $q$ vertices the Tardos bound is $\approx\frac{q^{3}}{2}$ and the Bj{\"o}rner-Lov\'{a}sz-Shor bound is $\approx q^{2}$ while Theorem \ref{thm:srg} yields a bound of $\approx{5q}$ which is quite close to $\frac{q}{2}$ which is (according to numerical evidence) probably the right answer.
\end{expl}

\begin{rmrk}
Our estimates are rather sharp since:
\begin{itemize}
\item
The Schl\"{a}fli graph is $SRG(27,16,10,8)$ and has $o=\frac{5}{972}>\frac{1}{216}=\frac{1}{nc}$.
\smallskip
\item
The triangular graph $T(21)$ (that is, the line graph of $K_{21}$) is $SRG(210,38,19,4)$ and has $f=\frac{4769}{176400} >\frac{1}{37}=\frac{1}{k-1}$.
\end{itemize}
\end{rmrk}

\section{Notes}
Some time after obtaining the main result of the paper (Theorem \ref{thm:main}) I have had the pleasure of reading the paper \cite{BakSho13} by Baker and Shokrieh and realized, with some surprise and not a little gratification, that they had - among other considerations in a remarkable and far-ranging work - carried out an analysis of the problem of game duration along similar lines, stressing the importance of generalized inverses (Moore-Penrose or others). Cf. especially their Section 3.2 and Remark 5.4. 

The crucial difference between the analysis in \cite{BakSho13} and the present paper is the 
observation that the off-diagonal entries of $L^{\dagger}$ have different dynamics from those of the diagonal entries and that this fact can be exploited to sharply reduce the bounds. 


\section{Acknowledgments}
I wish to thank Professors Michael Krivelevich and Farbod Shokrieh for valuable discussions on the subject of chip firing.

%
%
%


%


\bibliographystyle{abbrv}
\bibliography{nuim}
\end{document}